\documentclass[12pt,reqno,oneside]{amsart}
\usepackage{amsfonts,amssymb,latexsym,xspace,epsfig,graphics,color}
\usepackage{amsmath,enumerate,stmaryrd,xy}
\usepackage{bbm}
\usepackage{natbib}
\usepackage{subfigure}

\newcommand{\db}{\bar{d}}
\newcommand{\E}{\mathbb{E}}
\newcommand{\Q}{\mathbb{Q}}

\newcommand{\sig}{\sigma}

\newcommand{\R}{\ensuremath{\mathbb{R}}}     
\newcommand{\Z}{\ensuremath{\mathbb{Z}}}    
\renewcommand{\P}{\ensuremath{\mathbb{P}}}

\usepackage{algorithmic}
\usepackage{algorithm}

\newtheorem{lemma}{Lemma}

\newtheorem{theo}{Theorem}

\newtheorem{coro}{Corollary}

\begin{document}
\DeclareGraphicsExtensions{.pdf,.gif,.jpg}



\title[Explicit estimates in the Bramson-Kalikow model]{Explicit estimates in the Bramson-Kalikow model}



%
\author{Gallesco, C.$^1$}
\address{$^1$Departmento de Estat\'istica, Instituto de Matem\'atica, Estat\'istica e Ci\^encia de Computa\c{c}\~ao, Universidade de Campinas, Brasil}
\email{gallesco@ime.unicamp.br}

\author{Gallo, S.$^2$}
\address{$^2$Departamento de M\'etodos Estat\'isticos, Instituto de Matem\'atica, Universidade Federal de Rio de Janeiro, Brasil.}
 \email{sandro@im.ufrj.br}
 
 \author{Takahashi, D. Y.$^3$}
\address{$^3$Neuroscience Institute and Psychology Department, 
Princeton University, USA.}
\email{takahashiyd@gmail.com}

\thanks{This work is supported by USP project ``Mathematics, computation, language and the brain'', FAPESP project ``NeuroMat'' (grant 2011/51350-6),  CNPq projects ``Stochastic Modeling of the Brain Activity'' (grant 480108/2012-9), and ``Sistemas 
Estoc\'asticos: equil\'ibrio e n\~ao equil\'ibrio, limites em escala e percola\c c\~ao'' (grant 
474233/2012-0). SG was partially supported by FAPESP (grant 2009/09809-1). CG was partially supported by FAPESP (grant 2009/51139-3). DYT was partially supported by FAPESP grant 2008/08171-0 and Pew Latin American Fellowship.}


\maketitle

\begin{abstract}
The aim of the present article is to explicitly compute parameters for which the Bramson-Kalikow model exhibits phase-transition. The main ingredient of the proof is a simple new criterion for non-uniqueness of $g$-measures.  We show that the existence of multiple $g$-measures compatible with a  function $g$ can be proved by estimating the $\bar{d}$-distances between some suitably chosen Markov chains. The method is optimal for the important class of binary regular attractive functions, which includes the Bramson-Kalikow model.
\end{abstract}

\section{Introduction} \label{sec:Intro}

In this work we consider chains of infinite order, or equivalently $g$-measures, on a finite alphabet. They constitute an important class of stochastic models, which includes, for example, Markov chains,  stochastic models that exhibit non-uniqueness and models that are not Gibbsian \citep{fernandez/gallo/maillard/2011}. The question of uniqueness of $g$-measures was extensively studied and important progresses have been obtained in several areas related to probability and ergodic theory, from the seminal works of \cite{onicescu/mihoc/1935, doeblin/fortet/1937} to recent advances in \cite{johansson/oberg/pollicott/2012, gallo/paccaut/2013}, and the contributions of \cite{harris/1955, keane/1972, walters/1975,  lalley/1986, stenflo/2002, fernandez/maillard/2005} among many others. Notwithstanding, the problem of non-uniqueness is much less understood and the literature is still based on few examples \citep{bramson/kalikow/1993, hulse/2006, berger/hoffman/sidoravicius/2005}. As far as we know, general criteria for non-uniqueness have only been obtained for the class of regular attractive functions \citep{gallo/takahashi/2011,  hulse/1991}. 
 
In the present article we focus on the Bramson-Kalikow (BK) model \citep{bramson/kalikow/1993}. It is the most well studied example of non-uniqueness \citep{lacroix/2000, friedli/2010, gallo/takahashi/2011}, nevertheless our understanding of the model is still far from complete. For instance, to our knowledge, there is no explicit computation of the values of the parameters for which the BK model exhibits multiple $g$-measures (see \citet{friedli/2010} for related discussion).    Our main result (Theorem \ref{theo:BK}) gives such explicit  relationship between the parameters in the case of non-uniqueness.  Furthermore, we obtain an improvement on the range of parameters that imply non-uniqueness of the BK model. Corollaries \ref{prop:application} and  \ref{prop:application2} give numerical examples of choices for these parameters.  The proof of this result is based on three ingredients: (1) a new and simple criterion for non-uniqueness of $g$-measures (Theorem \ref{theo:perturbation}), (2) a concentration of measure inequality for $g$-measures obtained using a result from \citet{chazottes/collet/kulske/regig/2000}, and (3) $\bar{d}$-distance estimates using a coupling from the past algorithm. 

Theorem \ref{theo:perturbation} has a life of its own and is a criterion for non-uniqueness of $g$-measures that in principle can be applied to other models.
The motivation of Theorem \ref{theo:perturbation} is to avoid the direct study of functions $g$ with multiple $g$-measures as these are objects that are generally difficult to analyze \citep{gallo/takahashi/2011}. Instead, we study the properties of a sequence of suitably chosen Markov chains.  Theorem \ref{theo:perturbation} is inspired by the works of \citet{bramson/kalikow/1993}, \citet{lacroix/2000}, and \citet{hulse/1991}, but has the advantage of being formulated using the $\bar{d}$-distance, which is key to our constructive proof of Theorem  \ref{theo:BK}. Moreover, Theorem \ref{theo:attractive}  states that our criterion (Theorem \ref{theo:perturbation})  is optimal in the important class of binary regular attractive functions, giving a necessary and sufficient condition for non-uniqueness in this class, which includes the BK model. 

\vspace{0,2cm}

The article is organized as follows. We state the main results and relevant definitions in Section \ref{sec:main}. In Section \ref{sec:couplings} we introduce the couplings used to prove Theorem \ref{theo:BK} and Section \ref{sec:prooftheo2} contains the proofs of Theorem \ref{theo:BK}  and Corollaries \ref{prop:application} and \ref{prop:application2}. Finally, in Section  \ref{sec:proofs} we prove Theorem \ref{theo:perturbation} and Theorem \ref{theo:attractive}.

\section{Notation, definitions and main results} \label{sec:main}

Let $A$ be a finite set we call \emph{alphabet} and $\mathcal{X} = A^{\Z_-}$. We denote by $x_i$ the $i$-th coordinate of $x \in \mathcal{X}$ and for $i \leq j$ we write $x^{-i}_{-j}:=(x_{-i}\ldots x_{-j})$. For $x, y \in \mathcal{X}$, a \emph{concatenation} $x^{0}_{-i}y$ is a new sequence $z\in \mathcal{X}$ with $z^0_{-i} = x^{0}_{-i}$ and $z^{-i-1}_{-\infty} = y$. We introduce on $\mathcal{X}$ the metric $\rho(x,y) := \min\{\frac{1}{j+1}: x^0_{-j} = y^0_{-j} \}$, which turns $\mathcal{X}$ into a compact metric space. Denote by $\mathcal{B}$ the Borel $\sigma$-algebra on $\mathcal{X}$.  Let $T: \mathcal{X} \rightarrow \mathcal{X}$ be the \emph{shift operator} such that for $x \in \mathcal{X}$ we have $(Tx)_i = x_{i-1}$.  We denote by $\mathcal{C}(\mathcal{X})$ the space of continuous functions with norm $\|f\|:= \sup_{x \in \mathcal{X}}|f(x)|$. Let also 
 $$\mathcal{G}:=\{g\in\mathcal{C}(\mathcal{X}):g(x)\in(0,1)\,\,\textrm{and}\,\,\sum_{a\in A}g(ax)=1,\,\,\forall x\in \mathcal{X}\}.$$
 In the literature \citep{bramson/kalikow/1993}, a function in $\mathcal{G}$ is called \emph{regular}.
 We denote by $\mathcal{M}_k$ the set of regular $k$-th order \emph{Markov kernels} on $\mathcal{X}$ and by  $\mathcal{M} = \bigcup_{k \geq 0} \mathcal{M}_k$ the set of regular Markov kernels. We have $\mathcal{M} \subset \mathcal{G} $. 
 Sometimes, we consider a well ordered set $A$ and then $\mathcal{X}$ is endowed with partial order $x \geq y \Leftrightarrow x_i \geq y_i$ for all $i \in \Z_{-}$. A function $g \in \mathcal{G}$ is \emph{attractive} if $A$ is well ordered and for all $a\in A$, $\sum_{b \geq a}g(bx)$ is an increasing function of $x \in \mathcal{X}$.

Let $g \in \mathcal{G}$, following \cite{walters/1975} we say that a probability measure $\mu$ on $\mathcal{X}$ is a \emph{$g$-measure} if it is $T$-invariant and, for all $a \in A$ and $x \in \mathcal{X}$, $\mu(\{x \in \mathcal{X}: x_0 = a\}| T^{-1}\mathcal{B})(x) = g(ax^{-1}_{-\infty})$ or equivalently, 
\begin{equation}\label{eq:invariant}
\int_\mathcal{X} fd\mu = \int_\mathcal{X} \sum_{a \in A}g(ax)f(ax)d\mu,
\end{equation}
for all $f \in \mathcal{C}(\mathcal{X})$. A process $(X_n)_{n\in\Z}$ is said to be \emph{compatible with $g$} if its law is a $g$-measure. In this article, we are interested on conditions for \emph{non-uniqueness} of $g$-measures, \textit{i.e.}, sufficient conditions for the existence of several $g$-measures  with the same function $g$.\\

We will now define the model introduced by \cite{bramson/kalikow/1993}. Let $A = \{-1,+1\}$, $\epsilon \in (0,1/2)$, and $(m_j)_{j \geq 1}$ be an increasing sequence of positive odd numbers. Let $x \in \mathcal{X}$, we denote by $p_{[m_j]} \in \mathcal{M}_{m_j}$ the function
\begin{equation} \label{eq:BKpart}
p_{[m_j]}(x) = {\bf1}\left\{x_0\sum_{l=1}^{m_j}x_{-l} > 0\right\}(1-\epsilon) + {\bf1}\left\{x_0\sum_{l=1}^{m_j}x_{-l} < 0\right\}\epsilon.
\end{equation}
Let $(\lambda_j)_{j \geq 1}$ be a sequence of positive numbers such that $\sum_{j=1}^\infty \lambda_j = 1$. Given $(m_j)_{j \geq 1}$ and $(\lambda_j)_{j \geq 1}$, the BK-model is given by the function $p \in \mathcal{G}$ such that, for all $x \in \mathcal{X}$,
\begin{equation} \label{eq:BKdef}
p(x) = \sum_{j = 1}^\infty \lambda_j p_{[m_j]}(x).
\end{equation}
It is immediate that the BK-model $p$ is attractive and regular. \citet{bramson/kalikow/1993} showed that if $\lambda_j = (1-s)s^{j-1}$ for $s \in (2/3,1)$, there exists a sequence $(m_j)_{j \geq 1}$ for which the BK model has multiple $p$-measures. However, it is not known  how the sequence  $(m_j)_{j \geq 1}$ should be explicitly chosen.
Theorem \ref{theo:BK} below exhibits an explicit relationship between sequences $(\lambda_j)_{j \geq 1}$ and $(m_j)_{j\geq 1}$ for which there are multiple $p$-measures.

\begin{theo} \label{theo:BK}
Let $(\lambda_j)_{j \geq 1}$ and $(m_j)_{j \geq 1}$ be the sequences that define the BK model $p$ in (\ref{eq:BKdef}). Let $m_0 = 0$,  $r:\{1,2,\ldots\}\rightarrow \Z_+$ be a function such that $r_k < k$, and $\alpha\in(0,\frac{1}{2}-\epsilon)$. If for all $k \geq 0$ we have $\sum_{j\geq k+2}\lambda_j > \sum_{j=r_{k+1}+1}^{k+1}\lambda_{j}$ and 
\begin{align} \label{eq:theo2a}
m_{k+1} \geq \frac{A_k} {\left(\sum_{j\geq k+2}\lambda_j - \sum_{j=r_{k+1}+1}^{k+1}\lambda_{j}\right)^2},
\end{align}
where 
\begin{equation} \label{eq:theo2b}
A_k := 8\Big( 1-2\epsilon\Big)^{-2}(1+m_{r_{k+1}} (2\epsilon)^{-m_{r_{k+1}}})^{2}\ln\Big(2^{k+2}(1+m_k(2\epsilon)^{-m_k})\alpha^{-1}\Big),
\end{equation}
then the corresponding BK model $p$ has multiple $p$-measures. 
\end{theo}

Let us now give two numerical examples of sequences $(\lambda_j)_{j \geq 1}$ and $(m_j)_{j\geq 1}$ for which there are multiple $p$-measures, illustrating the relationship between the sequences $(\lambda_j)_{j\geq 1}$ and $(m_j)_{j\geq 1}$ in Theorem \ref{theo:BK}.

\begin{coro} \label{prop:application}
   Let $\epsilon=1/4$ and for $j\geq 1$, $\lambda_j=\frac{1}{2}\big(\frac{2}{3}\big)^j$. Let $m_1=217$, $c$ be an odd positive integer, and for $j\geq 1$,  $m_{j+1} = c^{m_j}$. If $c \geq 577$, then the associated BK model has multiple $p$-measures.
\end{coro}
 
 The next corollary illustrates the improvement on the growth rate of $(m_j)_{j\geq 1}$ that we obtain due to a better understanding of its relationship with the rate of $(\lambda_j)_{k \geq 1}$ through the function $r$ in Theorem \ref{theo:BK}.
  \begin{coro} \label{prop:application2}
   Let $\epsilon=1/4$ and for $j \geq 1$, $m_j = 2^{cj^2}-1$. Let $b_1=1$, $c \geq 0$,  and for $l \geq 2$, $b_l = 2^{(c\sum_{j=1}^{l-1}b_j)^2}$. For $l \geq 1$ and $j \in \{\sum_{k=1}^{l-1}b_{k}+1, \ldots, \sum_{k=1}^{l}b_{k}\}$ we set
$\lambda_j = (3/4)^{l-1}/(4b_l)$. If $ c \geq 8$,
 then the associated BK model has multiple $p$-measures.
 \end{coro}
 
 To prove that our result is tight,  we need a criterion for uniqueness of $p$-measures with conditions on the parameters comparable to Theorem \ref{theo:BK}. 
 Known criteria for uniqueness \citep{johansson/oberg/pollicott/2012, fernandez/maillard/2005} don't give such conditions. Therefore, the existence of a sharp transition from uniqueness to non-uniqueness regime for the BK model still remains an interesting open problem. \\

 Before stating Theorem \ref{theo:perturbation}, we need to introduce Orstein's $\bar{d}$-distance \citep{shields/1996}. We say that a measure $\nu$ on $\mathcal{X}\times \mathcal{X}$ is a coupling between $\mu$ and $\mu'$ if for all measurable subsets $\Gamma$ of $\mathcal{X}$ we have $\nu(\Gamma \times \mathcal{X}) = \mu(\Gamma)$ and $\nu(\mathcal{X} \times \Gamma) = \mu'(\Gamma)$. The set of all $T\otimes T$-invariant couplings between $\mu$ and $\mu'$ is denoted by $C(\mu,\mu')$ and the $\bar{d}$-distance between $\mu$ and $\mu'$ is defined by 
 \begin{equation*}
 \bar{d}(\mu,\mu') = \inf_{\nu \in C(\mu,\mu')}\nu(\{(x,x') \in \mathcal{X}\times \mathcal{X}: x_0 \neq x_0'\}).
 \end{equation*}
 
 Let $\tilde{\mu}$ and $\tilde{T}$ be respectively the natural extensions on $A^{\Z}$ of the $g$-measure $\mu$ and the shift operator $T$. We say that this natural extension $(\tilde{\mu},\tilde{T})$ is \emph{Bernoulli} if it is isomorphic to a Bernoulli shift.

\begin{theo} 
\label{theo:perturbation}
 Let $(g_j)_{j \geq 0}$ and $(g_j')_{j \geq 0}$ be two sequences of functions in $\mathcal{M}$ both converging to $g \in \mathcal{G}$ in $\mathcal{C}(\mathcal{X})$. Let $\mu_j$ and  $\mu_j'$ be the unique associated $g_j$ and  $g_j'$-measures. If there exists an integer $k\geq0$ such that 
 \begin{equation} \label{eq:perturbation}
 \sum_{j \geq k}\bar{d}(\mu_j, \mu_{j+1})  +  \sum_{j \geq k}\bar{d}(\mu_j', \mu_{j+1}')  < \bar{d}(\mu_k,\mu_k'),
\end{equation}
then there exist at least two distinct $g$-measures $\mu$ and $\mu'$. Moreover, the natural extensions of both $g$-measures are Bernoulli.
\end{theo}

The main advantage of Theorem \ref{theo:perturbation} is that we need to know nothing \textit{a priori} about the non-unique $g$-measures.  The only requirement is  a good control of the coupling between the Markov approximations.

We state below that the converse of Theorem \ref{theo:perturbation} holds for the important class of binary attractive functions $g\in \mathcal{G}$, which includes for example the Bramson-Kalikow model  \citep{bramson/kalikow/1993}. 

\begin{theo} \label{theo:attractive}
Let $A = \{-1, +1\}$. If $g \in \mathcal{G}$ is attractive, then there exist multiple $g$-measures if and only if there exist two sequences $(g_j)_{j \geq 0}$ and $(g_j')_{j \geq 0}$ of functions in $\mathcal{M}$ both converging to $g \in \mathcal{G}$ in $\mathcal{C}(\mathcal{X})$ such that the associated $g_j$ and $g_j'$-measures $\mu_j$ and $\mu_j'$ satisfy, for some $k \geq 0$, the inequality (\ref{eq:perturbation}).
\end{theo}

\section{Couplings and perfect simulations} \label{sec:couplings}

The proof of  Theorem \ref{theo:BK} will use Theorem \ref{theo:perturbation} which involves the $\bar{d}$-distance between Markov chains. Therefore, we will construct several couplings and Markov chains. The constructions are  conceptually straightforward but tedious to write, thus for convenience of the reader we will  define all the constructions in the present section. 

All the stationary measures  needed in the proof of Theorem \ref{theo:BK} will be simultaneously constructed using only a single sequence  ${\bf U}:=(U_j)_{j \in \Z}$ of i.i.d. r.v.'s uniformly distributed in $[0,1)$. We let $(\Omega,\mathcal{F},\mathbb{P})$ denote the probability space corresponding to the i.i.d. sequence ${\bf U}$, and $\mathbb{E}$ the expectation under $\mathbb{P}$. 

The first notion that we need to introduce is that of \emph{coupling from the past} (CFTP) algorithm.

\vspace{1cm}
\paragraph{{\bf Coupling from the past}}

The idea of the construction is the following. First, for $g \in \mathcal{G}$, we associate an \emph{update function} $F: [0,1) \times A^{\Z_-}\rightarrow A$ that satisfies $\mathbb{P}(F(U_{0},x)=a)=g(ax)$ for any $x\in \mathcal{X}$ and $a \in A$. 
For any pair of integers $i,j$ such that $-\infty<i\le j <+\infty$, let $F_{\{j,i\}}(U_{i}^{j},x)\in A^{j-i+1}$ be the sample obtained by \emph{applying recursively $F$ on the fixed past $x$}, i.e, let $F_{\{i,i\}}(U_{i},x):=F(U_{i},x)$ and for any $j > i$
\[
F_{\{j,i\}}(U_{i}^{j},x):=F(U_{j}, F_{\{j-1,i\}}(U_{i}^{j-1},x))F_{\{j-1,i\}}(U_{i}^{j-1},x).
\]
Secondly, define $F_{[i,i]}(U_{i},x):=F(U_{i},x)$ and 
\begin{equation}\label{eq:F}
F_{[j,i]}(U_{i}^{j},x)=F\left(U_{j}, F_{\{j-1,i\}}(U_{i}^{j-1},x)x\right).
\end{equation}
$F_{[j,i]}(U_{i}^{j},x)$ is the last symbol of the sample $F_{\{j,i\}}(U_{i}^{j},x)$.

With these definitions, for all $x \in \mathcal{X}$ we can construct the sequence $\left(X_{j}^{(x)} \right)_{j \geq 1}$ defined by
\[
F_{[j,1]}(U_{1}^{j},x):= X^{(x)}_{j},
\]
 which is the stochastic process  starting with a fixed past $x \in \mathcal{X}$ and updated according to $g$. 

Now we can define the notion of perfect simulation by coupling from the past. Let $\theta$ be the \emph{coalescence time}  defined by
\begin{equation}\label{eq:theta}
\theta:= \min\left\{i \geq 0:F_{[0,-i]}(U_{-i}^{0},x) = F_{[0,-i]}(U_{-i}^{0},y)\,\,\,\textrm{for all }\,x, y \in \mathcal{X}\right\}.
\end{equation}
It can be proved (see \cite{propp/wilson/1996, comets/fernandez/ferrari/2002, desantis/piccioni/2010} for instance) that if $\theta$ is $\mathbb{P}$-a.s. finite then there is a unique process $(X_j)_{j \in \Z}$ compatible with $g$, such that, 
\[
F_{[0, -\theta]}(U_{-\theta}^{0},x)\stackrel{\mathcal{D}}{=} X_{0}\,\,\,\,\forall x \in \mathcal{X}.
\]
Therefore, when an update function $F$ and a $\mathbb{P}$-a.s. finite $\theta$ exist, we say that there exists a CFTP algorithm that \emph{ perfectly simulates} $(X_j)_{j \in  \Z}$. Observe that we are considering the bi-infinite stationary process on $\Z$ rather than the process restricted on $\Z_+$, as this is more convenient for the proof of Theorem \ref{theo:BK}.

We note that \cite{gallo/takahashi/2011} proved that an attractive $g$-measure is unique if and only if it can be perfectly simulated through a CFTP algorithm. Therefore,  the non-unique $p$-measures for the BK model considered in the present paper cannot be simulated through a CFTP.  Instead, in the present article we use CFTP to make a simultaneous construction of all the Markov chains obtained by truncating the initial function $p$. These truncations are introduced in the next paragraph.

\vspace{1cm}
\paragraph{{\bf Update function for the truncating Markov kernels}} We will consider different Markov kernels in the proof of Theorem \ref{theo:BK}. They are truncations of order $m_{k}$ of the Bramson-Kalikow's function $p \in \mathcal{C}(\mathcal{X})$ defined in (\ref{eq:BKdef}). Let $x\in \mathcal{X}$, $p_{[0]} \in \mathcal{M}_{0}$ be defined by $p_{[0]}(x) = (1-\epsilon){\bf 1}\{x_0 > 0\}+\epsilon{\bf 1}\{x_0 < 0\}$, and $p_{[m_j]} \in \mathcal{M}_{m_j}$ be defined as in (\ref{eq:BKpart}). For $ l > k \geq 0$, consider the following $m_k$-th order Markov kernels
\begin{equation}\label{eq:MarkovApp1}
p_{k}(x) = \sum_{j = 1}^{k} \lambda_j p_{[m_j]}(x)+\sum_{j = k+1}^\infty \lambda_j p_{[0]}(x),
\end{equation}
\begin{equation} \label{eq:MarkovApp2}
p'_{k}(x) = \sum_{j = 1}^{k} \lambda_j p_{[m_j]}(x)+\sum_{j = k+1}^\infty \lambda_j (1-p_{[0]}(x)),
\end{equation}
\begin{equation} \label{eq:MarkovApp3}
q_{k,l}(x) = \sum_{j = 1}^{k} \lambda_j p_{[m_j]}(x)+\sum_{j = k+1}^l\lambda_{j}(1-p_{[0]}(x))+\sum_{j = l+1}^\infty \lambda_j p_{[0]}(x),
\end{equation}
\begin{equation} \label{eq:MarkovApp4}
q_{k,l}'(x) = \sum_{j = 1}^{k} \lambda_j p_{[m_j]}(x)+\sum_{j = k+1}^l \lambda_{j}p_{[0]}(x)+\sum_{j = l+1}^\infty \lambda_j (1-p_{[0]}(x)),
\end{equation}
where $\sum_{j = 1}^{0}x_j$ means that the summand is zero.

Defining $\bar{\lambda}_0 := 2\epsilon$ and $\bar{\lambda}_j:=\lambda_j(1-2\epsilon)$ for $j\geq 1$, we can respectively rewrite (\ref{eq:MarkovApp1}) and (\ref{eq:MarkovApp3}) as
\begin{equation*}
p_{k}(x)=\bar{\lambda}_0\frac{1}{2}+\sum_{j=1}^{k}\bar{\lambda}_{j}{\bf1}\left\{x_0\sum_{i=1}^{m_j} x_{-i} > 0\right\}+\sum_{j\ge k+1}\bar{\lambda}_{j}\left(\frac{1+x_0}{2}\right),
\end{equation*}
\begin{equation*}
q_{k,l}(x)=\bar{\lambda}_0\frac{1}{2}+\sum_{j=1}^{k}\bar{\lambda}_{j}{\bf1}\left\{x_0\sum_{i=1}^{m_j} x_{-i} > 0\right\}+\sum_{j= k+1}^l\bar{\lambda}_{j}\left(\frac{1+x_0}{2}\right)+\sum_{j\ge l+1}\bar{\lambda}_{j}\left(\frac{1-x_0}{2}\right).
\end{equation*}
Similar equations hold for (\ref{eq:MarkovApp2}) and (\ref{eq:MarkovApp4}).

Now, for any past $x \in \mathcal{X}$, consider the intervals 
\begin{equation} \label{eq:intervals}
I_0(-1):=[0,\,\epsilon[\,,\,\,\,I_0(+1)=[\epsilon,\,2\epsilon[\,\,\,\textrm{and}\,\,\,\,I_{j}=\left[\sum_{i=0}^{j-1}\bar{\lambda}_{i},\,\sum_{i=0}^{j}\bar{\lambda}_{i}\right[,  \,\,\,\,\,j\ge1.
\end{equation}
We observe that the lengths $|I_0(-1)|=|I_0(+1)|=\epsilon$ and for $j \geq 1$, $|I_{j}|=\bar{\lambda}_{j}$.

It is natural to consider the following update functions for the Markov kernels $p_{k}$ and $q_{k,l}$ respectively. 
\begin{align*}
F^{p_{k}}(U_{0},x)=&\,\sum_{a\in A}a{\bf 1}\{U_{0}\in I_0(a)\}+\sum_{a\in A}\sum_{j=1}^{k}a{\bf 1}\{U_{0}\in I_{j}\}{\bf1}\left\{a\sum_{ i=1}^{m_j} x_{-i} > 0\right\}\\
&+\sum_{a\in A}\sum_{j\ge k+1}a{\bf 1}\{U_{0}\in I_{j}\}\left(\frac{1+a}{2}\right),
\end{align*}
and
\begin{align*}
F^{q_{k,l}}(U_{0},x)=&\,\sum_{a\in A}a{\bf 1}\{U_{0}\in I_0(a)\}+\sum_{a\in A}\sum_{j=1}^{k}a{\bf 1}\{U_{0}\in I_{j}\}{\bf1}\left\{a\sum_{i=0}^{m_j-1} x_{-i} > 0\right\}\\
&+\sum_{a\in A}\sum_{j= k+1}^la{\bf 1}\{U_{0}\in I_{j}\}\left(\frac{1-a}{2}\right)+\sum_{a\in A}\sum_{j\ge l+1}a{\bf 1}\{U_{0}\in I_{j}\}\left(\frac{1+a}{2}\right).
\end{align*}
We can define analogous update functions  for $p_{k}'$ and $q_{k,l}'$. 

Let $\underline{+1}, \underline{-1} \in \mathcal{X}$ be defined by $\underline{+1}_j = 1$ and $\underline{-1}_j = -1$ for $j\leq 0$. We define the \textit{coalescence time}
\begin{align*}
\theta^{p_k}:&= \min\left\{i \geq 0:F^{p_k}_{[0,-i]}(U_{-i}^{0},x) = F^{p_k}_{[0,-i]}(U_{-i}^{0},y)\,\,\,\textrm{for all }\,x, y \in \mathcal{X}\right\}\\
&= \min\left\{i \geq 0:F^{p_k}_{[0,-i]}(U_{-i}^{0},\underline{+1}) = F^{p_k}_{[0,-i]}(U_{-i}^{0},\underline{-1})\right\},
\end{align*}
where the last equality is a direct consequence of the attractiveness of $p_k$. We substitute in the above definitions $p_k$  by $p'_k$, $q_{k,l}$, or $q_{k,l}'$ to define $\theta^{p'_k}$, $\theta^{q_{k,l}}$, and $\theta^{q_{k,l}'}$.

We also define, for any $i\in\mathbb{Z}$ and $k\ge1$, the \textit{regeneration time} of order $k$
\[
\eta_k:=\min\{i\geq m_{k}-1:U_{-j}\in I_0(-1)\cup I_0(+1)\,,\,\,j=i-m_k+1,\ldots,i\}.
\]

\vspace{.5cm}
\paragraph{{\bf Couplings between the chains and an upperbound for $\theta^{p_k}$ and $\eta_{k}$}}

We couple all the chains together constructing them simultaneously using the CFTP algorithm with same sequence ${\bf U}$ and the respective update functions. Consequently, the coupling law is always $\mathbb{P}$, \textit{i.e.}, the product law of ${\bf U}$. We also use the same symbol to indicate the marginal process and coupled process, when there is no ambiguity. 	

In what follows, we collect some lemmas that will be used in the proof Theorem \ref{theo:BK}.
Let us give an upper bound on the expectation of the coalescence and regeneration times that hold for $p_k$, $p'_k$, $q_k$, and $q_k'$. First, observe that by construction, 
\begin{equation*}
F^{p_k}_{[0,-\eta_{k}]}(U_{-\eta_{k}}^{0},\underline{+1}) = F^{p_k}_{[0,-\eta_{k}]}(U_{-\eta_k}^{0},\underline{-1})
\end{equation*}
$\P$-a.s. and, therefore,
\begin{equation*}
\P(\eta_{k} \geq \theta^{p_k}) = 1.
\end{equation*}
The same holds for $p'_k$, $q_{k,l}$, and $q_{k,l}'$.
Now, we have the following lemma.

\begin{lemma} \label{lemma:uppertime}
Let $\eta_{k}$ be the regeneration time of order $k$. We have that
\[
\mathbb{E}[\theta^{p_k}]
\leq \mathbb{E}[\eta_{k}]
\leq\frac{m_{k}}{(2\epsilon)^{m_{k}}}.
\]
The same bound holds for $\theta^{p_k'}$, $\theta^{q_{k,l}}$,  and $\theta^{q_{k,l}'}$.
\end{lemma}
\begin{proof} By the definition of $\eta_k$ we have
\[
\mathbb{P}(\eta_{k}\ge n.m_{k})\leq \prod_{i=1}^{n}\mathbb{P}\left(\left\{\bigcap^{-(i-1)m_{k}}_{j=-im_{k}+1}\left\{U_{j}\in \bigcup_{l=0}^{im_{k}-1+j}I_l  \right \}\right\}^{c}\right).
\]
Using the stationarity and independence of ${\bf U}$, we have for $i=1,\ldots,n$
\begin{align*}
\mathbb{P}\left(\bigcap^{-(i-1)m_{k}}_{j=-im_{k}+1}\left\{U_{j}\in \bigcup_{l=0}^{im_{k}-1+j}I_l \right\} \right)&=\prod_{j=1}^{m_{k}}\mathbb{P}\left(U_{j}\in \bigcup_{l=0}^{j}I_l \right)\\
&=\prod_{j=1}^{m_{k}}\mathbb{P}\left(U_{0}\in \bigcup_{l=0}^{j}I_l \right).
\end{align*}
A simple upper bound is $\prod_{j=1}^{m_{k}}\mathbb{P}(U_{0}\in \bigcup_{l=0}^{j}I_l )\leq (2\epsilon)^{m_k}$. This yields 
\[
\mathbb{E}[\eta_{k}] \leq m_{k}\sum_{n\ge1}\mathbb{P}(\eta_{k}\ge n.m_{k})\leq m_{k}\sum_{n\ge1}(1-(2\epsilon)^{m_k})^{n}\leq \frac{m_k}{(2\epsilon)^{m_{k}}}.
\]

\end{proof}

\begin{lemma}
\label{lemma:lowerbound}
Let $k < l$ and $(Y^{k,l}_j)_{j\in \Z}$ be the stationary process compatible with $q_{k,l}$. If  $\sum_{j \geq l+1}\lambda_j > \sum_{j=k+1}^{l}\lambda_{j}$ then 
\begin{equation}
\label{EspLow} 
\E[Y^{k,l}_0] \geq \Big(1-2\epsilon \Big)\left(\sum_{j \geq l+1}\lambda_j - \sum_{j=k+1}^{l}\lambda_{j}\right) > 0.
\end{equation}

\end{lemma}
\begin{proof}
 Let $(Z^{k,l}_j)_{j \in \Z}$ be the stationary process compatible with $q_{k,l}'$ we observe that
\begin{align*}
\E[Y^{k,l}_0] &= \P(Y^{k,l}_0 = 1)-\P(Y^{k,l}_0 = -1) \\
&= \P(Y^{k,l}_0 = 1)-\P(Z^{k,l}_0 = 1).
\end{align*}
Now, we want to construct a maximal coupling between $(Y^{k,l}_j)_{j\in \Z}$ and $(Z^{k,l}_j)_{j \in \Z}$. For this we define an update function for $q'_{k,l}$ using a set of intervals slightly different from the intervals defined in (\ref{eq:intervals}). We have
\begin{equation} \label{eq:intervals}
I'_0(-1):=[0,\,\epsilon[\,,\,\,\,I'_0(+1)=[\epsilon,\,2\epsilon[\,\,\,\textrm{and}\,\,\,\,I'_{j}=\left[\sum_{i=0}^{j-1}\bar{\lambda}_{i},\,\sum_{i=0}^{j}\bar{\lambda}_{i}\right[,  \,\,\,\,\,\text{for}\;\;k \geq j\geq 1,
\end{equation}
\begin{equation} \label{eq:intervals}
I'_{j}=\left[\sum_{i=0}^{j-1}\bar{\lambda}_{i} + \sum_{i\geq l+1}\bar{\lambda}_{i},\,\sum_{i=0}^{j}\bar{\lambda}_{i}+\sum_{i\geq l+1}\bar{\lambda}_{i}\right[,  \,\,\,\,\,\text{for}\;\; l \geq j\geq k+1,
\end{equation}
and
\begin{equation} \label{eq:intervals}
I'_{j}=\left[\sum_{i=0}^{k}\bar{\lambda}_{i} + \sum_{i = l+1}^{j}\bar{\lambda}_{i},\,\sum_{i=0}^{k}\bar{\lambda}_{i}+\sum_{i = l+1}^j\bar{\lambda}_{i}\right[,  \,\,\,\,\,\text{for}\;\; j\geq l,
\end{equation}
where $\sum_{i = l+1}^{l}\bar{\lambda}_{i} = 0$. The update function for $q'_{k,l}$  is then defined by
\begin{align*}
H^{q'_{k,l}}(U_{0},x)=&\,\sum_{a\in A}a{\bf 1}\{U_{0}\in I'_0(a)\}+\sum_{a\in A}\sum_{j=1}^{k}a{\bf 1}\{U_{0}\in I'_{j}\}{\bf1}\left\{a\sum_{i=0}^{m_j-1} x_{-i} > 0\right\}\\
&+\sum_{a\in A}\sum_{j= k+1}^{l}a{\bf 1}\{U_{0}\in I'_{j}\}\left(\frac{1-a}{2}\right)+\sum_{a\in A}\sum_{j\ge l+1}a{\bf 1}\{U_{0}\in I'_{j}\}\left(\frac{1+a}{2}\right).
\end{align*}
Observe that this update function is different from $F^{q'_{r,k+1}}$, which uses the intervals defined in (\ref{eq:intervals}).

By construction
\begin{equation*}
 \P(Y^{k,l}_0 = 1)-\P(Z^{k,l}_0 = 1) = \P(Y^{k,l}_0 \neq Z^{k,l}_0).
\end{equation*}

Now, the following lower bound is an immediate consequence of the construction of the coupling
\begin{equation*}
\P(Y^{k,l}_0 \neq Z^{k,l}_0) \geq \Big(1-2\epsilon \Big)\left(\sum_{j \geq l+1}\lambda_j - \sum_{j=k+1}^{l}\lambda_{j}\right).
\end{equation*}

\end{proof}

\section{Proof of Theorem \ref{theo:BK} and Corollaries \ref{prop:application}, \ref{prop:application2}} \label{sec:prooftheo2}

The proof of Theorem \ref{theo:BK} is based on the results of the last section, Theorem \ref{theo:perturbation} which is proved in the next section, and the following lemma.


\begin{lemma}
\label{Lemconc}
Let $k < l$  and $(Y^{k,l}_j)_{j\in \Z}$ be the stationary process compatible with $q_{k,l}$. For all $l>k > 0$ and $j \geq 1$, we have
\begin{equation} \label{eq:concfund}
\P\left(\left|\frac{1}{m_{l}}\sum_{i = 1}^{m_{l}}Y^{k,l}_{i}  -\E[Y^{k,l}_0]\right| \geq  \frac{\E[Y^{k,l}_0]}{2}\right) \leq  2\exp\left(-\frac{m_{l}\E[Y^{k,l}_0]^2}{8\Big(1+\E[\theta^{q_{k,l}}]\Big)^2}\right).
\end{equation}
\end{lemma}

\begin{proof}

We will use Theorem 1 of \citet{chazottes/collet/kulske/regig/2000} to obtain an upper bound for the left-hand side of (\ref{eq:concfund}).
Let $(Z_j)_{j \in \Z}$ be a canonical process on $\{-1,+1\}^{\Z}$ with law $\mu$. Let $(Z_j^{(+1\sigma)})_{j \geq 1}$ and $(Z_j^{(-1\sigma)})_{j \geq 1}$ be respectively the processes with laws defined by the conditional distributions $\mu((Z_j)_{j\geq 1}=\cdot \mid Z_{0}=1, Z_{-1}=\sig_{-2},\dots, Z_{-i+1}= \sig_{-i})$ and $\mu((Z_j)_{j \geq 1}=\cdot \mid Z_{0}=-1, Z_{-1}=\sig_{-2},\dots, Z_{-i+1}=\sig_{-i})$.
We denote by $\Q^{\sig}_{i}$ the maximal coupling between the conditional distributions. Now, we introduce the upper-triangular matrix $D^{\sig}$ defined for $1\leq i<j \leq n$ by
\begin{align}
D^{\sig}_{i,i}&:=1\nonumber\\
D^{\sig}_{i,j}&:= \Q^{\sig}_{i}\Big( Z^{(+1\sigma)}_{j} \neq Z^{(-1\sigma)}_{j}  \Big).
\end{align}
Then, we define the matrix $\bar{D}$ as $\bar{D}_{i,j}:=\sup_{\sig\in \{-1,1\}^{n}}D^{\sig}_{i,j}$.
For a given function $f: \{-1,1\}^{n}\to \R$ we define the variation of $f$ at site $i$ with $1\leq i\leq n$ by 
\[
\delta_if:= \sup_{\sig_j=\sig'_j, i\neq j}|f(\sig)-f(\sig')|.
\]

Now, let $n\geq 1$ be arbitrary and assume that $\|\bar{D}\|_{2}<\infty$ and $\|\delta f\|_{2}<\infty$. Then, Theorem 1 in \citet{chazottes/collet/kulske/regig/2000} states that, for all functions $f: \{-1,1\}^{n}\to \R$ (with a slight abuse of notation, we also consider $f$ as a function from $\{-1,1\}^{\Z}\to \R$ which depends only on the first $n$ positive coordinates)  and all $t>0$, we have 
\begin{align}
\label{Collet}
\mu(|f-\E[f]|\geq t)\leq 2 \exp\left(-\frac{2t^2}{\|\bar{D}\|^2_{2}\|\delta f\|^2_{2}}\right).
\end{align} 

In our case, for the process $(Y^{k,l}_j)_{j \in \Z}$ and measure $\P$, we observe that the elements of matrix $\bar{D}_{i,j}$ are bounded from above by the probabilities $\P(\theta^{q_{k,l}} > j)$ for all $j>i\geq 1$. 
To see this we note that
\begin{align*}
\P(\theta^{q_{k,l}} > j) &= \P\left(F^{q_{k,l}}(U_{-i}^{0},\underline{+1}) \neq F^{q_{k,l}}(U_{-i}^{0},\underline{-1}) \;\;\;\text{for}\;\; i =1, \ldots, j  \right)\\
&=\P\left(F^{q_{k,l}}(U_{-j}^{0},\underline{+1}) \neq F^{q_{k,l}}(U_{-j}^{0},\underline{-1}) \right),
\end{align*}
where the last equality is a consequence of the attractiveness of $q_r$.
Now, by the stationarity of ${\bf U}$ we have
\begin{align*}
\P(\theta^{q_{k,l}} > j) &= \P\left(F^{q_{k,l}}_{[j,0]}(U_{0}^{j},\underline{+1}) \neq F^{q_{k,l}}_{[j,0]}(U_{0}^{j},\underline{-1}) \right)\\
&\geq \bar{D}_{i,j}.
\end{align*}
By Cauchy-Schwarz's inequality, we obtain
\begin{align*}
\|\bar{D}u\|^2
=\sum_{i=1}^{n}\Big( \sum_{j=i}^{n}\bar{D}_{i,j}u_j\Big)^2
&=\sum_{i=1}^{n}\Big( \sum_{j=i}^{n}\bar{D}^{1/2}_{i,j}(\bar{D}^{1/2}_{i,j}u_j)\Big)^2\nonumber\\
&\leq \sum_{i=1}^{n}\Big(\sum_{j=1}^{n} \bar{D}_{i,j} \Big)  \Big(\sum_{j=1}^{n} u_j^2 \bar{D}_{i,j}   \Big)\nonumber\\
&\leq \Big(1+\sum_{j=1}^{n}\P(\theta^{q_{k,l}} > j)\Big)\sum_{i=1}^{n}\sum_{j=i}^{n}u_j^2\bar{D}_{i,j}\nonumber\\
&\leq \Big(1+\sum_{j=1}^{n}\P(\theta^{q_{k,l}} > j)\Big)^2\|u\|^2_2
\end{align*}
for all $u\in \R^n$. Taking $n=m_{l}$, we deduce that
\begin{align}
\label{norm1}
\|\bar{D}\|^2_{2}\leq \left(1+\sum_{j=1}^{m_{l}}\P(\theta^{q_{k,l}} > j)\right)^2 \leq \left(1+\E[\theta^{q_{k,l}} ]\right)^2.
\end{align}

Now, taking $f=f\Big(x_1,\dots, x_{m_{l}}\Big)= \frac{1}{m_{l}}\sum_{i=1}^{m_{l}}x_i$ we have $\delta_i f=\frac{2}{m_{l}}$ if $i\in \{1,\dots,m_{l}\}$. Thus, we obtain 
\begin{align}
\label{norm6}
\|\delta g\|^2_2=\sum_{i=1}^{m_{l}}\left(\frac{2}{m_{l}}\right)^2= \frac{4}{m_{l}}.
\end{align}
Applying (\ref{Collet}) and using (\ref{norm1}) and (\ref{norm6}), we obtain
\begin{equation*}
\P\left(\left|\frac{1}{m_{l}}\sum_{j=1}^{m_{l}}Y^{k,l}_j -\E[Y^{k,l}_0]\right| \geq  \frac{\E[Y^{k,l}_0]}{2} \right) \leq  2\exp\left(-\frac{m_{l}\E[Y^{k,l}_0]^2}{8\Big(1+\E[\theta^{q_{k,l}}]\Big)^2}\right).
\end{equation*}
\end{proof}

\begin{proof}[Proof of Theorem \ref{theo:BK}]
\mbox{}\\*

We fix a sequence $(\lambda_j)_{j \geq 1}$ of positive real numbers such that $\sum_{j \geq 1} \lambda_j =1$. Let $r: \{1,2, \ldots\} \rightarrow \Z_+$ such that $r_k < k$  and $\sum_{j\geq k+1}\lambda_j > \sum_{j=r_k+1}^{k}\lambda_{j}, \forall k\geq 1$. The sequence of odd positive integer numbers $(m_j)_{j\geq 1}$ will be chosen afterwards.

Clearly $(p_j)_{j \geq 1}$ and $(p'_j)_{j \geq 1}$ defined in (\ref{eq:MarkovApp1}) and (\ref{eq:MarkovApp2}) converge to the Bramson-Kalikow's $p$ in $\mathcal{C}(\mathcal{X})$.
For all $k\geq 0$, let $\mu_k$ (resp. $\mu'_k$)  be the unique stationary measure compatible with $p_k$ (resp.\ $p'_k$). Observe that for $k=0$, $\mu_0$ (resp. $\mu'_0$) is  a Bernoulli process of parameter $1-\epsilon$ (resp. $\epsilon$).

We will apply Theorem \ref{theo:perturbation} with $k=0$. Since $\bar{d}(\mu_0,\mu'_{0})=1-2\epsilon$, we need to find an explicit sequence $(m_j)_{j \geq 1}$ such that
\begin{equation}
\label{DERT0}
\sum_{k\geq 0}\bar{d}(\mu_k,\mu_{k+1})+\sum_{k\geq 0}\bar{d}(\mu'_k,\mu'_{k+1})< 1-2\epsilon.
\end{equation}
 By symmetry of the kernels $p_k$ and $p'_k$, (\ref{DERT0}) is equivalent to
\begin{equation}
\label{DERT}
2\sum_{k\geq 0}\bar{d}(\mu_k,\mu_{k+1})< 1-2\epsilon.
\end{equation}

Now, our task is to upper bound $\bar{d}(\mu_k,\mu_{k+1})$. For all $k \geq 0$, let $\left(X_j^{k}\right)_{j\in \Z}$ be the stationary process compatible with the measure $\mu_k$.
By definition of the $\db$-distance we have that 
\begin{align}
\label{D1}
\db(\mu_k, \mu_{k+1})\leq \P\left(X^k_0\neq X^{k+1}_0\right),
\end{align}
where $\P$ is the coupling defined in Section \ref{sec:couplings}.
Define for all $i\in \Z_-$, the interval $I_i:=[i-m_{k+1},i-1]$ and the events
\[
S_i:=\left\{\sum_{j\in I_i}X^{k+1}_j>0\right\}.
\]
 As $\eta_k$ (defined as in Lemma \ref{lemma:uppertime}) is a stopping time for the filtration $(\mathcal{F}_i)_{i\geq 0}=(\sigma(U_0,U_{-1},\dots, U_{-i}))_{i\geq 0}$ and the events $S_i$ are independent of  $\mathcal{F}_i$ for all $i\geq 0$, we have by construction of the coupling $\P$ and Wald's equality
\begin{align}
\label{D2}
\P(X^k_0\neq X^{k+1}_0)
= \P\left(\bigcup_{i=0}^{\eta_{k}}S^c_i\right)
\leq \E\left[\sum_{i=0}^{\eta_k} \mathbf{1}_{\{(S_i)^c\}}\right]
= (\E[\eta_k]+1) \P(S_0^c).
\end{align}
Combining (\ref{D1}) and (\ref{D2}) we obtain
\begin{align}
\label{ERB}
\bar{d}(\mu_k,\mu_{k+1})\leq (\E[\eta_k]+1) \P(S_0^c)
\end{align}
for all $k\geq 0$.

To obtain an upper bound for $\E[\eta_k]$ we use Lemma \ref{lemma:uppertime}; for $\P(S_0^c)$ we proceed as follows.
Let $r:=r_{k+1}$ and $(Y_j^{r,k+1})_{j \in \Z}$ be the process compatible with $q_{r,k+1}$. Observe that for all $k \geq r \geq 0$ we have $q_{r,k+1} \in \mathcal{M}_r$.  Also note that, for any $n \geq 1$, and integers $l_1, \ldots, l_n$, we have by construction that 
$$\P\left(\bigcup_{j = 1}^{n}\left\{X_{l_j}^{k+1} < Y_{l_j}^{r,k+1}\right\}\right) = 0,$$
 and therefore
\begin{equation}
\label{Concent0}
\P\left(\sum_{j=1}^{m_{k+1}} X^{k+1}_{j}< 0 \right) \leq\P\left(\sum_{j=1}^{m_{k+1}} Y^{r,k+1}_{j}< 0 \right).
\end{equation}
Furthermore, we have
\begin{equation}
\label{Concent1}
\P\left(\sum_{j=1}^{m_{k+1}} Y^{r,k+1}_{j}  < 0 \right) \leq  \P\left(\left|\frac{1}{m_{k+1}}\sum_{j=1}^{m_{k+1}} Y^{r,k+1}_{j} -\E[Y^{r,k+1}_0]\right| \geq  \frac{\E[Y^{r,k+1}_0]}{2} \right),
\end{equation}
and therefore we can upper bound $\P(S_0^c)$ using a concentration of measure inequality for a Markov chain of order $r < k+1$.\\
Combining   (\ref{ERB}), (\ref{Concent0}),  (\ref{Concent1}),  and Lemmas \ref{lemma:lowerbound} and \ref{Lemconc}, we deduce that, for all $k\geq 0$,
\begin{equation*}
\bar{d}(\mu_k,\mu_{k+1}) \leq 2(\E[\eta_k]+1)\exp\left(-\frac{m_{k+1}\left(\sum_{j\geq k+2}\lambda_j - \sum_{j=r+1}^{k+1}\lambda_{j}\right)^2\Big(1-2\epsilon  \Big)^2}{8\Big(1+\E[\theta^{q_{r,k+1}}]\Big)^2}\right).
\end{equation*}

Let $\alpha>0$ such that $\alpha<\frac{1}{2}-\epsilon$. Define 
\[
A_0:= 8\Big( 1-2\epsilon \Big)^{-2} \ln\left(4\alpha^{-1}\right)
\]
and for all $k\geq 1$,
\begin{equation*}
\label{AAAKKKK}
A_k := 8\Big( 1-2\epsilon\Big)^{-2}(1+m_r (2\epsilon)^{-m_r})^{2}\ln\Big(2^{k+2}(1+m_k(2\epsilon)^{-m_k})\alpha^{-1}\Big).
\end{equation*}
Then, for all $k \geq 0$ choose $m_{k+1}$ as the first odd integer such that 
\begin{align*} \label{eq:condition}
m_{k+1} \geq \frac{A_k} {\left(\sum_{j\geq k+2}\lambda_j - \sum_{j=r+1}^{k+1}\lambda_{j}\right)^2}.
\end{align*}
With these choices, using Lemma \ref{lemma:uppertime} we obtain
\begin{equation*} \label{eq:desiredbar}
\bar{d}(\mu_k,\mu_{k+1})  \leq  \frac{\alpha}{2^{k+1}}
\end{equation*}
for all $k\geq 0$.
Since $\alpha<\frac{1}{2}-\epsilon$ we obtain (\ref{DERT}), which proves the theorem.
\end{proof}

\vspace{1cm}
\begin{proof}[Proof of Corollary \ref{prop:application}]
\mbox{}\\*

 If, for $j\geq 1$, we choose $\lambda_j=\frac{1}{2}\big(\frac{2}{3}\big)^j$, we have for $k \geq 1$, $\sum_{j\geq k+1}\lambda_j - \lambda_{k} \geq 0$, \textit{i.e.}, we have a function $r$ in (\ref{eq:theo2a}) defined by $r_k = k-1$. Let $\epsilon=1/4$ and $\alpha=1/8$, then $A_0=160\ln 2$ and, by (\ref{eq:condition}), $m_1$ must be chosen greater than $320\big(\frac{3}{2}\big)^2\ln 2\approx 216,74$. Let us take $m_1=217$. Now, from (\ref{AAAKKKK}), we can see that in this case, the sequence $(m_k)_{k\geq 1}$ must satisfy $m_k\geq k+1$ for all $k\geq 1$. Therefore, for $k\geq 1$, we have 
\begin{align*}
\frac{A_k} {\left(\sum_{j\geq k+2}\lambda_j - \lambda_{k+1}\right)^2}
&\leq 512 \Big(\frac{9}{2}\Big)^{k+1}(1+m_k2^{m_k})^3\nonumber\\
&\leq 512 \Big(\frac{9}{2}\Big)^{k+1} (64)^{m_k}\nonumber\\
&\leq (577)^{m_k}.
\end{align*}
\end{proof}

\vspace{.25cm}
\begin{proof}[Proof of Corollary \ref{prop:application2}]
\mbox{}\\*

Let $b_1 = 1$ and $c$ a positive constant to be fixed afterwards. For $l \geq 2$, we define $b_l = \lceil 2^{(c\sum_{j=1}^{l-1}b_j)^2} \rceil$, where $\lceil \cdot \rceil$ is the ceilling function.
 Let $s = 3/4$,  for $l \geq 1$ and $j \in \{\sum_{k=1}^{l-1}b_{k}+1, \ldots, \sum_{k=1}^{l}b_{k}\}$ we define
$$\lambda_j = \frac{s^{l-1}-s^l}{b_l}.$$
 It is straightforward to verify that $\sum_{j\geq 1}\lambda_j = 1$.  Let $r_k =  \left \lfloor \sqrt{\log(k)/c} \right \rfloor $ where $\log$ is base 2 logarithm and $\lfloor \cdot \rfloor$ is the floor function. We observe that by construction, for $l \geq 1$, we have
\begin{equation*}
 \sum_{j \geq k+1}\lambda_j - \sum_{j = r(k)+1}^{k}\lambda_j  \geq s^l-(s^{l-2}\wedge 1 -s^l) = \frac{1}{8}\left(\frac{3}{4}\right)^{l-2}.
\end{equation*}

We set $m_j = \lfloor 2^{cj^2}  \rfloor$ if $\lfloor 2^{cj^2}  \rfloor$ is odd, otherwise $m_j=\lfloor 2^{cj^2}  \rfloor-1$ . We want to obtain a sequence $(m_j)_{j \geq 1}$ that satisfies (\ref{eq:theo2a}) and (\ref{eq:theo2b}). Let
\begin{equation*}
B_k = 4(k+1)^2 2^{2(k+1)(\log(2\epsilon)^{-1})}.
\end{equation*}
 We have the following upper bound for (\ref{eq:theo2b}):
\begin{align*}
&A_k\leq \frac{8(k+2)\ln \frac{2}{\alpha}}{( 1-2\epsilon)^{2}}B_k+\frac{8\ln\Big(1+2^{ck^2}(2\epsilon)^{-2^{ck^2}}\Big)}{( 1-2\epsilon)^{2}}B_k.
\end{align*}
 Now, taking $\epsilon = 1/4$ and $\alpha = 1/8$, we have,
 \begin{align*}
A_k&\leq 128\ln2 B_k(k+2)+32\ln2B_k(ck^2+1) + 32 B_k 2^{ck^2}\\
&\leq 81B_k2^{ck^2}.
\end{align*}

Also, we observe that for $c\geq 8$, 
$B_k \leq 2^{2+2\log(k+1)+2(k+1)}\leq 2^{ck} $,
and, therefore,
\begin{equation*}
A_k \leq 81\cdot2^{c(k^2+k)}.
\end{equation*}
Also for $c\geq 2$, we have $\frac{1}{8}\left(\frac{3}{4}\right)^{l-2} \geq \left(\frac{1}{2}\right)^{l+1} \geq 2^{-ck}$.
Finally, to satisfy the conditions in Theorem \ref{theo:perturbation}, it is enough that 
 \begin{equation*}
2^{c(k+1)^2} \geq 81\cdot2^{c(k^2+2k)}.
\end{equation*}
The above inequality is satisfied if 
$c \geq 8$.

\end{proof}

\vspace{0.5cm}
\section{Proof of Theorems \ref{theo:perturbation} and \ref{theo:attractive} } \label{sec:proofs}

\begin{proof}[Proof of Theorem \ref{theo:perturbation}] 
\mbox{}\\*

We proceed in three main steps. First, we prove the existence of  a subsequence $(\mu_{v_j})_{j\geq 0}$ that converges in $\bar{d}$ to a measure $\mu$ compatible with $g$. The same naturally holds for a subsequence $(\mu'_{u_j})_{j \geq 0}$ and some measure $\mu'$ compatible with $g$. Then we prove that under the conditions of the theorem $\mu$ and $\mu'$ are actually distinct. The statement about bernoullicity then follows directly from the well-known fact that $\bar{d}$-limit of regular Markov chains are Bernoulli.

For the first step, we will prove that there exists a subsequence $(\mu_{v_j})_{j\geq 0}$  converging weakly  and in entropy to a measure $\mu$ compatible with $g$. Because regular Markov processes are finitely determined and, for this class of processes, the weak convergence and the convergence of the entropy together imply convergence in $\bar{d}$, we conclude that $\mu_{v_j}$ converges to $\mu$ in $\bar{d}$-distance (see definition in p.221 and Theorem IV.2.9 of \cite{shields/1996}). 

We consider processes on finite alphabet, therefore the space of respective probability measures endowed with the weak topology is compact. Hence, for any sequence $(\mu_j)_{j \geq 0}$ there exists a convergent subsequence $(\mu_{v_j})_{j\geq 0}$. Let $\mu$ be its weak limit. From the weak convergence of $\mu_{v_j}$ to $\mu$ and the convergence of $g_{v_j}$ to $g$ in $\mathcal{C}(\mathcal{X})$, it is immediate that $\mu$ is a $g$-measure.

Now, we observe that the entropy $H(\mu_{v_j})$ of an ergodic Markov process $\mu_{v_j}$can be written as
\begin{equation*}
H(\mu_{v_j}) = - \int_{\mathcal{X}}\log g_{v_j} d\mu_{v_j} .
\end{equation*}
A standard computation shows that the entropy $H(\mu)$ of $\mu$ is given by
\begin{equation*}
H(\mu) = - \int_{\mathcal{X}}\log g d\mu.
\end{equation*}
We note that $g \in \mathcal{G}$ and therefore $\log g \in \mathcal{C}(\mathcal{X})$. Again, because $\mu_{v_j} \rightarrow \mu$ weakly and $\log g_{v_j}$ converges to $\log g$ in $\mathcal{C}(\mathcal{X})$ we have that
\begin{equation*}
 \int_{\mathcal{X}}\log g_{v_j} d\mu_{v_j} \rightarrow \int_{\mathcal{X}}\log g d\mu
\end{equation*}
Thus, we conclude that $\mu_{v_j}$ converges in  $\bar{d}$ to $\mu$.
 
We now come to the second step, and prove that the limits $\mu$ and $\mu'$ are distinct. Taking $v_0 = k$, we have
\begin{equation*}
\lim_{j \to \infty}\bar{d}(\mu_k,\mu_{v_j}) = \bar{d}(\mu_k,\mu).
\end{equation*}
We also have
\begin{equation*}
\bar{d}(\mu_k,\mu_{v_j}) \leq \sum_{j=k}^\infty \bar{d}(\mu_{j},\mu_{j+1})
\end{equation*}
and, therefore,
\begin{equation*}
\bar{d}(\mu_k,\mu) \leq \sum_{j=k}^\infty \bar{d}(\mu_{j},\mu_{j+1}).
\end{equation*}
Similarly,
\begin{equation*}
\bar{d}(\mu_k',\mu') \leq \sum_{j=k}^\infty \bar{d}(\mu_{j}',\mu_{j+1}').
\end{equation*}
Thus, if (\ref{eq:perturbation}) is satisfied, we have 
\begin{equation*}
\bar{d}(\mu_k,\mu) + \bar{d}(\mu_k',\mu') < \bar{d}(\mu_k,\mu_k')
\end{equation*}
showing that there exist two distinct $g$-measures $\mu$ and $\mu'$.

\end{proof}

\vspace{0.5cm}
\begin{proof}[Proof of Theorem \ref{theo:attractive}]
\mbox{}\\*

The proof follows from the properties of attractive functions $g \in \mathcal{G}$ and associated $g$-measures described in \citet{hulse/2006, hulse/1991}.  

For $f \in \mathcal{C}(\mathcal{X})$, $g \in \mathcal{G}$, and all $x \in \mathcal{X}$, we define the \emph{Ruelle operator} $L_g$ by
\begin{equation*}
L_gf(x) = \sum_{a \in A}g(ax)f(ax).
\end{equation*}
Let $\underline{+1}, \underline{-1} \in \mathcal{X}$ defined by $\underline{+1}_i = 1$ and  $\underline{-1}_i = -1$ for $i \leq 0$. By Lemma 2.1 in \citet{hulse/1991}, if $g$ is attractive, for $f \in \mathcal{C}(\mathcal{X})$, we have
\begin{equation} \label{eq:limitmaxmim}
\lim_{n \rightarrow \infty}L_g^nf(\underline{+1}) = \int_\mathcal{X} fd\mu^+,
\end{equation}
\begin{equation*}
\lim_{n \rightarrow \infty}L_g^nf(\underline{-1}) = \int_\mathcal{X} fd\mu^-,
\end{equation*}
where $\mu^+$ and $\mu^-$ are extremal $g$-measures. If $\mu^+=\mu^-$ we have a unique $g$-measure.
Let $h_1$ and $h_2$ be elements of $\mathcal{G}$ and $x, y \in \mathcal{X}$. We say that $h_1$ \emph{dominates} $h_2$ if for all $x\geq y$  we have $h_1(1x) \geq h_2(1y)$.  
From \citet{hulse/2006} (p.442) if $h_1$ dominates $h_2$,  for any increasing function $f \in \mathcal{C}(\mathcal{X})$, $x \geq y$, and $n\geq 1$ we have
 \begin{equation} \label{eq:dominance}
 L_{h_1}^nf(x) \geq L_{h_2}^nf(y).
\end{equation}

We start with the following lemma.
\begin{lemma} \label{lemma:dbarattract}
 Let $h_1, h_2$ be attractive and $h_1$ dominates $h_2$.   If $\nu_1^+$ and $\nu_2^+$ are the extremal $h_1$ and $h_2$-measures defined respectively by iterating $L_{h_1}$ and $L_{h_2}$ as in (\ref{eq:limitmaxmim}), we have that
 \begin{equation}
 \bar{d}(\nu_1^+, \nu_2^+) = \nu_1^+(\{x\in \mathcal{X}:x_0 = 1\})-\nu_2^+(\{x\in \mathcal{X}:x_0 = 1\}).
\end{equation}
 \end{lemma}

\begin{proof}

As in \citet{hulse/1991}, for $x, y \in \mathcal{X}$, $a,b \in A$, and $h_1,h_2 \in \mathcal{C}(\mathcal{X})$, we define the function $P: \mathcal{X} \times \mathcal{X} \to [0,1]$ by  
\begin{equation*}
P(ax,by)=
\left\{
\begin{array}{ccc}
\min \left\{ h_1(ax), h_2(ay) \right\}&\textrm{if }\,a=b\\
\left\{h_1(ax)-h_2(ay)\right\}\vee0&\textrm{otherwise},
\end{array}
\right.
\end{equation*}
and $\sum_{a\in A}\sum_{b\in A} P(ax,by)=1$.
Let $f \in \mathcal{C}(\mathcal{X})$. We define $\nu_1^+$ and $\nu_2^+$ by $\lim_{n \rightarrow \infty}L_{h_1}^nf(\underline{+1}) = \int_\mathcal{X} fd\nu_1^+$ and $\lim_{n \rightarrow \infty}L_{h_2}^nf(\underline{+1}) = \int_\mathcal{X} fd\nu_2^+$ , respectively. If $h_1$ and $h_2$ are attractive and $h_1$ dominates $h_2$, we can use $P$ to define a coupling between $\nu_1^+$ and $\nu_2^+$. To see this, let $f_1, f_2 \in \mathcal{C}(\mathcal{X})$ and $x,y \in \mathcal{X}$. We introduce the Ruelle operator $L_{P}$ as
\begin{equation*}
L_{P}(f_1\otimes f_2)(x,y) = \sum_{a \in A} \sum_{b \in A}P(ax, by)f_1(ax)f_2(by).
\end{equation*}
For any functions $f_1, f_2 \in \mathcal{C}(\mathcal{X})$,  we have that $\lim_{n \rightarrow \infty}L^n_{P}(f_1\otimes f_2)(\underline{+1},\underline{+1})$ exists and defines a coupling $\nu$ between $\nu_1^+$ and $\nu_2^+$  (\citet{hulse/1991}). By construction and definition of $P$, this coupling has the property that $\nu(\{(x,y) \in \mathcal{X}\times \mathcal{X}: x_0 < y_0\}) = 0$. This implies that
\begin{align*}
\bar{d}(\nu_1,\nu_2) &\leq \nu(\{(x,y) \in \mathcal{X}\times \mathcal{X}: x_0 \neq y_0\}) \notag \\
&=\nu_1(\{x\in \mathcal{X}:x_0 = 1\})-\nu_2(\{x\in \mathcal{X}:x_0 = 1\}).
\end{align*}
Moreover, we also have by definition of $\bar{d}$-distance that 
\begin{equation*}
 \bar{d}(\nu_1^+,\nu_2^+) \geq \nu_1^+(\{x\in \mathcal{X}:x_0 = 1\})-\nu_2^+(\{x\in \mathcal{X}:x_0 = 1\}),
\end{equation*}
which implies that 
\begin{equation*} 
\bar{d}(\nu_1^+,\nu_2^+) = \nu_1^+(\{x\in \mathcal{X}:x_0 = 1\})-\nu_2^+(\{x\in \mathcal{X}:x_0 = 1\}).
\end{equation*}

\end{proof}

 	We introduce a sequence of functions $g_{j}, g_{j}' \in \mathcal{M}$ for each $j\geq 1$ and $x \in \mathcal{X}$ by
\begin{equation*}
g_j(1x^{-1}_{-\infty}) = \sup_{y \in \mathcal{X}}g(1x^{-1}_{-j}y),
\end{equation*}
and
\begin{equation*}
g_j'(1x^{-1}_{-\infty}) = \inf_{y \in \mathcal{X}}g(1x^{-1}_{-j}y).
\end{equation*}
For $j=0$ we define $g_0(1x^{-1}_{-\infty})= \sup_{y \in \mathcal{X}}g(1y)$ and $g_0'(1x^{-1}_{-\infty})= \inf_{y \in \mathcal{X}}g(1y)$. Observe that if $g$ is attractive, $g_j$ and $g_j'$ are also attractive. Moreover, for all $j \geq 0$, $g_j$ dominates $g$ and $g$ dominates $g_j'$.

Let $\mu_j$ be the unique $g$-measure of $g_j$. From Lemma \ref{lemma:dbarattract}
\begin{equation} \label{eq:attra}
\bar{d}(\mu_j,\mu^+) = \mu_j(\{x\in \mathcal{X}:x_0 = 1\})-\mu^+(\{x\in \mathcal{X}:x_0 = 1\}).
\end{equation}

Now, by definition of $\mu^+$, for any $g$-measure $\mu$ we have
\begin{equation} \label{eq:ineqmax}
\mu^+(\{x\in \mathcal{X}:x_0 = 1\})-\mu(\{x\in \mathcal{X}:x_0 = 1\}) \geq 0.
\end{equation}
Let $\mu_{v_j}$ be any subsequence  converging weakly to some $g$-measure. Because, for all $j\geq 1$, $g_{v_j}$ dominates $g$, we have
\begin{equation*}
\lim_{j \rightarrow \infty}\mu_{v_j}(\{x\in \mathcal{X}:x_0 = 1\})-\mu^+(\{x\in \mathcal{X}:x_0 = 1\}) \geq 0.
\end{equation*}
The above equation together with (\ref{eq:ineqmax}) implies that 
\begin{equation*}
\lim_{j \rightarrow \infty}\mu_{v_j}(\{x\in \mathcal{X}:x_0 = 1\})=\mu^+(\{x\in \mathcal{X}:x_0 = 1\}).
\end{equation*}
As this holds for any subsequence, we have that
\begin{equation*}
\lim_{j \rightarrow \infty}\mu_{j}(\{x\in \mathcal{X}:x_0 = 1\})=\mu^+(\{x\in \mathcal{X}:x_0 = 1\}).
\end{equation*}

The above equation and Lemma \ref{lemma:dbarattract} imply
$$\lim_{j \rightarrow \infty}\bar{d}(\mu_j,\mu^+) = 0,$$
and we conclude that
\begin{equation} \label{eq:attractsup}
\sum_{j = k}^{\infty}\bar{d}(\mu_j,\mu_{j+1})  = \mu_k(\{x\in \mathcal{X}:x_0 = 1\})-\mu^+(\{x\in \mathcal{X}:x_0 = 1\}).
\end{equation}

Now, let $\mu_j'$ be the $g_j'$ measures. Repeating again the above arguments, we have that
\begin{equation} \label{eq:attractinf}
\sum_{j = k}^{\infty}\bar{d}(\mu_j',\mu_{j+1}')  = \mu^{-}(\{x\in \mathcal{X}:x_0 = 1\})-\mu_k'(\{x\in \mathcal{X}:x_0 = 1\}).
\end{equation}
Using Lemma  \ref{lemma:dbarattract} again, we have
\begin{equation} \label{eq:attractend}
\bar{d}(\mu_k,\mu_k')  = \mu_k(\{x\in \mathcal{X}:x_0 = 1\})-\mu_k'(\{x\in \mathcal{X}:x_0 = 1\}).
\end{equation}

Combining (\ref{eq:attractsup}), (\ref{eq:attractinf}), and (\ref{eq:attractend}) we have that inequality
\begin{equation*} 
\sum_{j = k}^{\infty}\bar{d}(\mu_j,\mu_{j+1})  + \sum_{j = k}^{\infty}\bar{d}(\mu_j',\mu_{j+1}')  < \bar{d}(\mu_k,\mu_k')
\end{equation*}
is equivalent to 
\begin{equation} \label{eq:condhulse}
\mu^+(\{x\in \mathcal{X}:x_0 = 1\})-\mu^-(\{x\in \mathcal{X}:x_0 = 1\}) > 0.
\end{equation}

Finally, from Theorem 2.2 in \citep{hulse/1991}, we have that inequality (\ref{eq:condhulse}) holds if and only if there are several $g$-measures.

\end{proof}

\bibliographystyle{jtbnew}
\bibliography{sandro_bibli.bib}

\end{document}